\documentclass[11pt,a4paper]{amsart}[2000]

\title[Asymptotic Dimension of Box Spaces]{The Asymptotic Dimension of Box Spaces for Elementary Amenable Groups}

\author{Martin Finn-Sell}
\address{Universit{\"a}t Wien, Fakult\"{a}t f\"{u}r Mathematik,  \phantom{-----------------------------------------}\linebreak \text{}\hspace{3.5mm} Oskar-Morgenstern-Platz 1,   1090 Wien, \"{O}sterreich }
\email{martin.finn-sell@univie.ac.at}

\author{Jianchao Wu}
\address{Westf{\"a}lische Wilhelms-Universit{\"a}t, Fachbereich Mathematik, \phantom{---------------------------}\linebreak \text{}\hspace{3.5mm} Einsteinstrasse 62, 48149 M{\"u}nster, Deutschland}
\email{jianchao.wu@uni-muenster.de}

\thanks{\emph{Supported by:} Martin Finn-Sell is partially supported by the ERC grant ANALYTIC 259527, and Jianchao Wu is supported by ERC Advanced Grant ToDyRiC 267079}

\usepackage{amsmath}
\usepackage{amssymb}
\usepackage{amsthm}
\usepackage[all]{xypic}
\usepackage{hyperref}

\allowdisplaybreaks

\theoremstyle{plain}
\newtheorem{Thm}{Theorem}[section]

\newtheorem{thmx}{Theorem}

\newtheorem{thm}[Thm]{Theorem}
\newtheorem{lem}[Thm]{Lemma}
\newtheorem{cor}[Thm]{Corollary}
\newtheorem{prop}[Thm]{Proposition}

\theoremstyle{definition}
\newtheorem{defn}[Thm]{Definition}
\newtheorem{ques}[Thm]{Question}

\newtheorem{rmk}[Thm]{Remark}

\newcommand{\N}{{\mathbb N}}
\newcommand{\Z}{{\mathbb Z}}

\numberwithin{equation}{section}

\newcommand{\asdim}{\mathrm{as}\dim}
\newcommand{\dist}{d}
\newcommand{\set}[1]{\left\{#1\right\}}
\newcommand{\diam}{\mathrm{diam}}

\begin{document}

\begin{abstract}
We show that the asymptotic dimension of box spaces behaves (sub)additively with respect to extensions of groups. As a result, we obtain that for an elementary amenable group, the asymptotic dimension of any of its box spaces is bounded above by its Hirsch length. This bound is shown to be an equality for a large subclass of groups including all virtually polycyclic groups.  
\end{abstract}
\maketitle

\section{Introduction}\label{sec:introduction}

Let $G$ be a residually finite group, and let $\sigma = ( G_{\alpha} )_{\alpha \in I}$ be a family of finite-index normal subgroups of $G$. Following Roe (\cite[Definition 11.24]{Roe2003} ), we may associate to this data a coarse space, $\square_\sigma G$, called the \emph{box space} of $G$ with respect to $\sigma$, defined to be the disjoint union of the corresponding finite quotients of $G$ together with some appropriate $G$-invariant coarse structure. The importance of the box space construction can be seen in the bridge it provides between coarse geometric properties of $\square_\sigma G$ on the one hand and group theoretic properties of $G$ on the other. This can be illustrated by a result of Guentner \cite[Proposition 11.39]{Roe2003}, which states: a residually finite discrete group $G$ is amenable if and only if the box space $\square_{\sigma} G$ has property A. There are now similar results for a-T-menability \cite{MR3105001,FinnSell20143758} and property (T) \cite{Willett:2014aa}. Another major application of the box space construction is to produce spaces that demonstrate certain coarse properties are not equivalent: a prominent example is the result of Arzhantseva, Guentner and \v{S}pakula \cite{MR2899681}, where they produce a box space of a free group with bounded geometry that does not have property A but does admit a coarse embedding into Hilbert space.

One novelty of our treatment is that we do not require $\sigma$ to only consist of normal subgroups. With this generality, one needs a condition that rules out certain pathological examples where the coarse structure of the box space $\square_\sigma G$ behaves badly in comparison with that of $G$ itself. The natural requirement turns out to be that the family $\sigma$ is \emph{semi-conjugacy-separating}, which asks not just for the ability to separate finite sets from the identity, but also to be able to separate entire conjugacy classes of finitely many elements from the identity. This condition guarantees that the family of quotient maps $\set{G \to G/G_{\alpha}}_{G_{\alpha} \in \sigma}$ accommodates arbitrarily large injective radii (see Section \ref{sec:preliminaries-coarse-geometry}). It is also general enough to cover the classical case of normal subgroups as any separating family of normal finite-index subgroups is semi-conjugacy-separating. 

A fundamental yet vastly useful invariant for coarse spaces is the \emph{asymptotic dimension} (see \cite{Gromov1993} and \cite{BD2008}). Many studies have been carried out to determine or bound the asymptotic dimension of groups, where the group itself is viewed as a coarse space by fixing a proper, right invariant metric on the group (e.g a (weighted) word-length metric associated to a generating set of the group). \emph{Finite} asymptotic dimension for a group has far-reaching consequences, for instance any group of finite asymptotic dimension satisfies the Novikov conjecture (\cite{Yu1998}). 

Related to this, a natural question to ask is when the box space of a residually finite group has finite asymptotic dimension. By the result of Guentner mentioned above (c.f.~ \cite[Proposition 11.39]{Roe2003}), such a group is necessarily amenable. This question has recently found application in the study of the Rokhlin dimension and the nuclear dimension of crossed product C*-algebras (\cite{SWZ2014}). To the best of our knowledge, this problem is far less well-understood than the previous one regarding the asymptotic dimension of the group itself. Beyond the easy cases of abelian groups and crystalographic groups, only recently was it shown that for a finitely-generated virtually nilpotent group, there exists a semi-conjugacy-separating family $\sigma$ of finite-index subgroups such that the asymptotic dimension of the box space $\square_\sigma G$ is finite (\cite{SWZ2014}). 

Given a group $G$, in order to show that any {box space} $\square_\sigma G$ has finite asymptotic dimension, it suffices to show it in the case when $\sigma$ consists of \emph{all} finite-index subgroups. We call the box space associated to this choice of $\sigma$ the \emph{extended full box space}, and denote it by $\square_\mathrm{F} G$. Similarly, if one is only interested in showing that for any separating family $\sigma$ of \emph{normal} finite-index subgroups, $\square_\sigma G$ has finite asymptotic dimension, then it suffices to look at the case when $\sigma$ consists of all normal finite-index subgroups. The resulting box space in this case is called the \emph{full box space}, and denoted by $\square_\mathrm{f} G$.  

In this paper, we develop an inductive method that allows us to estimate the asymptotic dimension of box spaces for a large class of groups. In particular, we prove:

\begin{thmx}\label{main-thm:inductive}
 Let $1 \to N \to G \to K \to 1$ be a short exact sequence of discrete groups. Then 
 \begin{align*}
  \asdim (\square_\mathrm{F} G) \leq & \ \asdim (\square_\mathrm{F} N) + \asdim (\square_\mathrm{F} K) \; ;  \\
  \asdim (\square_\mathrm{f} G) \leq & \ \asdim (\square_\mathrm{f} N) + \asdim (\square_\mathrm{f} K) \; .
 \end{align*}
\end{thmx}

Recall that a group $G$ is \textit{solvable} if there is a finite length composition series $\lbrace G_{i+1} \vartriangleleft G_{i} \rbrace_{i}$ for the group $G$ in which the factor groups $G_{i}/G_{i+1}$ are abelian, and \textit{elementary amenable} if $G$ can be constructed from finite groups and abelian groups through the processes of direct unions and extensions.

Let $\mathbf{Elem_{RF}}$ be the class of residually finite elementary amenable groups. It contains all residually finite virtually solvable groups and, in particular, all virtually polycyclic groups. These classes are readily accessible using traditional group theoretic methods and there are many numerical invariants that capture the complexity of such a group $G$ in terms of the abelian factors that appear. One prominent invariant is the \textit{Hirsch length}, $h(G)$, which for an abelian group is defined to be its dimension as a rational vector space and then extended to elementary amenable groups using properties for extensions and unions (see Section \ref{sect:elem}).

Applying Theorem \ref{main-thm:inductive}, we are able to produce an upper bound for the asymptotic dimension of any box space for a group in $\mathbf{Elem_{RF}}$.

\begin{thmx}\label{main-thm:bound}
 For any $G \in \mathbf{Elem_{RF}}$ and any semi-conjugacy-separating family $\sigma$ of finite-index subgroups of $G$, we have 
 \[
  \asdim(G) \leq \asdim(\square_\sigma G) \leq \asdim(\square_f G) \leq \asdim(\square_F G)  \leq h(G).
 \]
\end{thmx}

In particular, this upper bound gives a proof that elementary amenable groups have asymptotic dimension bounded above by their Hirsch length and we give a direct and short proof of this as Theorem \ref{thm:elem-induction} in the text. 

In the situation where $G$ is a semidirect product of a locally finite group by a virtually polycyclic group, we obtain:

\begin{thmx}\label{main-thm:calculate}
 For any residually finite $G \in \mathbf{LFin}\rtimes\mathbf{VPoly}$ and any semi-conjugacy-separating family $\sigma$ of finite-index subgroups of $G$, we have 
 \[
  \asdim(\square_\sigma G) = \asdim(G) = h(G).
 \]
\end{thmx}

This lower bound relies on a result of Dranishnikov and Smith (\cite[Theorem 3.5]{DS2006}) as well as a general principle of extensions by locally finite groups that we prove in Proposition \ref{prop:lf-ext}.

Other than locally finite groups and virtually polycyclic groups, this theorem also applies to all residually finite groups in the class $\mathbf{Fin} \wr \mathbf{VPoly}$, which consists of all reduced wreath products $F\wr H$ where $F$ is finite and $H \in \mathbf{VPoly}$. In this case, residual finiteness is equivalent to $F$ being abelian (Theorem 3.2 of Gruenberg \cite{MR0087652}). Thus the previous theorem implies that the group $F\wr H$ and any box space $\square_\sigma (F\wr H)$ associated to a semi-conjugacy-separating family $\sigma$ have asymptotic dimension equal to the Hirsch length of $H$ (see Corollary \ref{cor:wreath-product}). 

Notice that as a consequence of Theorem \ref{main-thm:calculate}, we have shown that for any residually finite group in $\mathbf{LFin}\rtimes\mathbf{VPoly}$, the asymptotic dimensions of its box spaces do not depend on the choice of the semi-conjugacy-separating family of finite-index subgroups. We do not know whether this holds in general. 

\begin{ques}
 Is it true that for any discrete residually finite group $G$ and any two semi-conjugacy-separating families $\sigma$ and $\sigma'$ of finite-index subgroups, we have $\asdim(\square_\sigma G) = \asdim(\square_{\sigma'} G)$?
\end{ques}

Theorem \ref{main-thm:calculate} answers the above question positively for any residually finite $G \in \mathbf{LFin}\rtimes\mathbf{VPoly}$, while for any non-amenable $G$, the answer is also affirmative because 
\begin{equation*}
 \asdim(\square_\sigma G) = \infty = \asdim(\square_{\sigma'} G).
\end{equation*}
Hence in order to find counterexamples for the above question, one needs to examine amenable groups outside of $\mathbf{LFin}\rtimes\mathbf{VPoly}$. This leads us to the following two question:

\begin{ques}
 Does Theorem \ref{main-thm:calculate} hold for all $G \in \mathbf{Elem_{RF}}$?
\end{ques}

\begin{ques}
 Is there an example of a (necessarily amenable) residually finite group outside of $\mathbf{Elem_{RF}}$ such that at least one box space of it has finite asymptotic dimension?
\end{ques}

\section*{Acknowledgements}

The authors are grateful to Joachim Zacharias for bringing them together at the beginning of this project. Additionally, the second author would like to thank Arthur Bartels for helpful discussions.

\section{Preliminaries: Asymptotic dimension and elementary amenable groups}\label{sec:preliminaries-group-theory}
In this section we recall the definition of asymptotic dimension for a metric space and state the permanence results we require from the literature. We finish the section with a permanence result for split extensions by locally finite groups.

\setcounter{subsection}{-1}

\subsection{Notations}\label{subsection:notations}
For a metric space $(X,d)$, a point $x \in X$, a subset $Y\subseteq X$ and a positive number $R$, we let $B_R(x;d) := \{y\in X \ | \ d(x,y) \leq R \}$ denote the closed $R$-neighbourhood of $x$, and $P(Y;R;d) := \{y\in X \ | \ d(Y,y) \leq R \}$ denote the closed $R$-neighbourhood of $Y$ in $X$. We often drop the metric $d$ in these notations when there is no danger of confusion.

\subsection{Asymptotic dimension}
Let $X$ be a metric space. Recall that a covering $\mathcal{U}$ of $X$
\begin{enumerate}
 \item is \emph{$S$-bounded} if there exists a constant $S>0$ such that $diam(U) \leq S$ for every $U \in \mathcal{U}$,
 \item is \emph{uniformly bounded} if it is $S$-bounded for some $S$,
 \item has \emph{multiplicity} at most $n$ if the intersection of any $(n+1)$ members of $\mathcal{U}$ is empty, and 
 \item has \emph{Lebesgue number} at least $R$ if any non-empty subset of $X$ with diameter no more than $R$ is contained entirely in some $U\in \mathcal{U}$.
\end{enumerate}

\begin{defn}\label{defn:asymptotic-dimension}
A family of metric spaces $\mathcal{X}= \lbrace X_{\alpha} \rbrace_{\alpha}$ has \textit{asymptotic dimension} $\leq n$ \textit{uniformly} if for any $R>0$ there exists $S>0$ such that every $X_{\alpha}$ has an $S$-bounded covering $\mathcal{U^{\alpha}}$ of Lebesgue number at least $R$ and multiplicity at most $n+1$. We write $\asdim(\mathcal{X}) \leq n$.

We say $\mathcal{X}$ has asymptotic dimension $= n$ uniformly, and write $\asdim(\mathcal{X}) = n$, if $n$ is the least natural number $N$ such that $\asdim(\mathcal{X}) \leq N$. We write $\asdim(\mathcal{X}) = \infty$ if no such $N$ exists.
\end{defn}

We remark that there are many other definitions of asymptotic dimension that are equivalent to this one (c.f.~ \cite{BD2008} for a summary of results and equivalent definitions), and the definition generalises to abstract coarse spaces (\cite[Definition 9.4]{Roe2003}). Asymptotic dimension satisfies the full range of coarse permanence properties (see for example \cite{EG-permanence}). A straightforward result of this nature is that asymptotic dimension does not increase when passing to subspaces. Moreover, for an uncountable family of metric spaces, the asymptotic dimension is determined by its countable subfamilies:

\begin{lem}\label{lem:asdim-countable}
 Let $\mathcal{X} = \lbrace (X_{\alpha}, d_{\alpha} ) \rbrace_{\alpha \in I}$ be a family of metric spaces. Then
 \[
  \asdim (\mathcal{X}) = \sup_{\mathrm{countable}\ J \subset I} \asdim (\lbrace (X_{\alpha}, d_{\alpha} ) \rbrace_{\alpha \in J}) \;.
 \]
\end{lem}

\begin{proof}
 ``$\geq$'' follows directly from definition, while ``$\leq$'' is a consequence of a reformulation of $\asdim(\mathcal{X} ) > n$ for any $n \in \N$: there exists $R > 0$ such that for any $S \in \Z^+$, there exists an index $\alpha_S \in I$ such that $X_{\alpha_S}$ does not possess any $S$-bounded covering of Lebesgue number at least $R$ and multiplicity at most $n+1$; whence for the countable subset $J=\lbrace \alpha_S \rbrace_{S \in \Z^+}$ we have $\asdim(\lbrace (X_{\alpha}, d_{\alpha} ) \rbrace_{\alpha \in J} ) > n$. 
\end{proof} 

For our purpose we also require from the literature the ``Hurewicz-type" theorems that deal with the asymptotic dimension of a ``fibred'' space. 

\begin{defn}
\begin{enumerate}
\item A family of maps $\lbrace f_{\alpha}: X_{\alpha} \to Y_{\alpha} \rbrace_{\alpha}$ between families of metric spaces is said to be \emph{Lipschitz} if there is $\lambda \geq 0$ such that each $f_{\alpha}$ is \emph{$\lambda$-Lipschitz}, i.e.~ $d(f_{\alpha}(x), f_{\alpha}(x')) \leq \lambda \: d(x, x')$ for any $x, x' \in X_{\alpha}$. 
\item A family of maps $\lbrace f_{\alpha}: X_{\alpha} \to Y_{\alpha} \rbrace_{\alpha}$ between families of metric spaces is said to be \emph{large scale uniform} if there exists a function $c:\mathbb{R}_{\geq 0} \rightarrow \mathbb{R}_{\geq 0}$ such that for all $\alpha$, $d_{X_{\alpha}}(x,x')\leq R$ implies $d_{Y_{\alpha}}(f_{\alpha}(x), f_{\alpha}(x'))\leq c(R)$.
\end{enumerate}
\end{defn}

The first Hurewicz-type theorem for asymptotic dimension we consider was proved by Dranishnikov and Bell \cite{MR2231870}.

\begin{thm}[{\cite[Theorem 1]{MR2231870}}]\label{prop:cf}
Let $f:X \rightarrow Y$ be a Lipschitz map from a geodesic metric space to a metric space. Suppose that, for every $R>0$, the family of coarse fibres $\lbrace f^{-1}(B_{R}(y)) \rbrace_{y \in Y}$ has asymptotic dimension $\leq n$ uniformly. Then
 \begin{equation*}
  \asdim(X) \leq \asdim(Y) + n.
 \end{equation*}\qed
\end{thm}

This theorem was later generalised by Brodskiy, Dydak, Levin and Mitra \cite{Brodskiy01062008}. Combining Theorem 4.11 and Corollary 4.12 of \cite{Brodskiy01062008}, we obtain:

\begin{thm}\label{prop:cf2}
 Let $f: X \to Y$ be a large-scale uniform function between metric spaces, such that, for every $R>0$, the family of coarse fibres $\lbrace f^{-1}(B_{R}(y)) \rbrace_{y \in Y}$ has asymptotic dimension $\leq n$ uniformly.  Then 
 \begin{equation*}
  \asdim(X) \leq \asdim(Y) + n.
 \end{equation*}\qed
\end{thm}

As a group quotient map is always Lipschitz for the quotient metric and the coarse fibre is a translate of a neighbourhood of the kernel (and thus coarsely equivalent to the kernel) we have the following natural corollary for group extensions, which was proved in \cite{MR2231870} when $G$ is finitely generated and \cite{DS2006,Brodskiy01062008} in the case of countable $G$.

\begin{cor}\label{cor:ext}
Let $1 \rightarrow N \rightarrow G \rightarrow K \rightarrow 1$ be a short exact sequence of groups. Then $\asdim G \leq \asdim N + \asdim K$. \qed
\end{cor}

It is immediate from definition that the asymptotic dimension of a locally finite group $G$ is $0$, because after equipping $G$ with a proper right-invariant metric, one has that, given any $R>0$, $G$ splits into a union of mutually $R$-separated cosets of a finite subgroup. Hence Corollary \ref{cor:ext} implies that if $\asdim K = n$ and $N$ is locally finite, then $\asdim G \leq n$. The following proposition enables us to retain a lower bound when the extension is split.

\begin{prop}\label{prop:lf-ext}
Let $1 \rightarrow N \rightarrow G \rightarrow K \rightarrow 1$ be a split short exact sequence where $N$ is locally finite. Then $\asdim G = \asdim K$.
\end{prop} 
\begin{proof}
From Corollary \ref{cor:ext} we know that $\asdim G \leq \asdim K$. However, because the extension is split, we know that $K$ is a subgroup of $G$, and whence $\asdim K \leq \asdim G$.
\end{proof}

\subsection{Solvable groups}

In this section we outline the necessary group theoretic definitions and results concerning the classes of polycyclic and solvable groups and the notion of Hirsch length for these groups.

\begin{defn}
A group $G$ is \textit{solvable} if there is a chain of subgroups $ 1 = G_{0} \vartriangleleft G_{1} \vartriangleleft ... \vartriangleleft G_{n-1} \vartriangleleft G_{n} = G$, such that the factor groups $G_{i}/G_{i-1}$ are abelian. 
\end{defn}

Recall that given a class of groups $\mathcal{C}$, a group $G$ is said to be virtually $\mathcal{C}$ (denoted $G \in \mathbf{V}\mathcal{C}$) if it contains a finite-index subgroup in $\mathcal{C}$. Let $\mathbf{VSolv}$ be the class of virtually solvable groups, which is closed under extensions and subgroups. It clearly contains the classes $\mathbf{VAb} \subset \mathbf{VNilp}$ of virtually abelian, and virtually nilpotent groups respectively. We note that the class of virtually polycyclic groups $\mathbf{VPoly}$ is contained in the class of virtually solvable groups, and contains the finitely generated groups in $\mathbf{VNilp}$.

It is known that the asymptotic dimension of solvable groups is controlled by the Hirsch length of the group:

\begin{defn}
Let $G$ be a solvable group. Then the \textit{Hirsch length of $G$} is defined as $h(G) := \sum_{i} dim_{\mathbb{Q}}(G_{i}/G_{i-1}\otimes \mathbb{Q})$.
\end{defn}

A nice observation, due to Dranishnikov and Smith \cite[Theorem 3.2]{DS2006}, is that for an abelian group $A$ one has $\asdim(A) = h(A) = dim_{\mathbb{Q}}(A\otimes \mathbb{Q})$. Coupling this with the definition of the Hirsch length, we get an alternative formula for a solvable group:
\begin{equation*}
h(G) = \sum_{i} \asdim (G_{i}/G_{i-1}).
\end{equation*}

In \cite{DS2006}, Dranishnikov and Smith also were able to conclude the following:

\begin{thm}\label{thm:DS}
Let $G$ be a virtually solvable group, then $\asdim(G) \leq h(G)$. If in addition $G$ is virtually polycyclic then this bound is sharp, that is $\asdim(G) = h(G)$.\qed
\end{thm}

By Proposition \ref{prop:lf-ext} we also obtain a sharp bound for any $G$ that is a semidirect product of a locally finite group by a virtually polycyclic group. This includes all the wreath products $F \wr H$, where $F$ is finite and $H$ is virtually polycyclic. 

We would like to remark that Hirsch length is often called \textit{torsion free rank} in the solvable group literature, and there are other natural definitions one could make to define the rank of a solvable group - such as the $p$-rank for a fixed prime $p$, which is defined for an abelian group $A$ to be the size of maximal independent set of elements that each have order $p$ and denoted $r_{p}(A)$. This rank is extended naturally to solvable groups by summing this rank over the abelian factor groups. A solvable group has \textit{finite rank} when:
\begin{equation*}
r(G) = h(G) + \sum_{p} r_{p}(G) < \infty.
\end{equation*}
Observe that torsion elements in an abelian group generate a locally finite subgroup, which as we remarked has asymptotic dimension $0$. From this we can see that the $p$-rank does not influence the asymptotic dimension.

On the other hand, being finite rank does enable a characterisation of residual finiteness in the class of solvable groups, as by Robinson \cite[Theorem A]{MR0228586}, a solvable group of finite rank is residually finite if and only if it has no divisible subgroups (such as $\mathbb{Q}$ or the quasi-cyclic $p$ group $\mathbb{Z}[\frac{1}{p}]/\mathbb{Z}$).

\subsection{Elementary amenable groups}\label{sect:elem}

In this section we describe the appropriate notion of Hirsch length for elementary amenable groups and give a generalisation of the Dranishnikov-Smith theorem to the class of elementary amenable groups.

\begin{defn}
A group $G$ is \textit{elementary amenable} if it is obtained from abelian groups and finite groups using the group constructions of taking extensions and taking direct limits. This can be described more systematically using transfinite induction (c.f.~ \cite{MR1143191}), but first we need some notions of class operations for groups.
\end{defn}

Let $\mathcal{C}$, $\mathcal{D}$ be classes of groups, and let $L\mathcal{C}$ denote the class of \textit{locally} $C$ groups, i.e the class of groups $G$ that have all finitely subgroups in $\mathcal{C}$; and let $\mathcal{C}\mathcal{D}$ be the class of $\mathcal{C}-$by$-\mathcal{D}$ groups, that is the class of groups $G$ such that there exists a normal subgroup $N \triangleleft G$ such that $N\in \mathcal{C}$ and $G/N \in \mathcal{D}$. 

To define the class of elementary amenable groups this way, begin with $\mathfrak{X}_{0}=\lbrace 1 \rbrace$ and $\mathfrak{X}_{1}$ the class of virtually abelian groups. Assuming the definition $\mathfrak{X}_{\alpha}$ we obtain $\mathfrak{X}_{\alpha+1}$ by the formula: $\mathfrak{X}_{\alpha+1}=(L\mathfrak{X}_{\alpha})\mathfrak{X}_{1}$. To complete the definition (by transfinite induction) define, for any limit ordinal $\beta$, $\mathfrak{X}_{\beta}=\cup_{\alpha < \beta} \mathfrak{X}_{\alpha}$ and finally $\mathbf{Elem}=\cup_{\alpha}\mathfrak{X}_{\alpha}$.

We can naturally use this framework to extend Hirsch length. Define the function $h$ first on $\mathfrak{X}_{1}$ by $h(A)=dim_{Q}(A\otimes \mathbb{Q})$ and inductively by supposing that it is defined for $G \in \mathfrak{X}_{\alpha}$ and then for $H \in L\mathfrak{X}_{\alpha}$ by
\begin{equation*}
h(H) = \sup \lbrace h(G) \mid G \in \mathfrak{X}_{\alpha}\rbrace.
\end{equation*}
When $G \in \mathfrak{X}_{\alpha+1}$ it has a normal subgroup $N$ in $L\mathfrak{X}_{\alpha}$ with $G/N$ in $\mathfrak{X}_{1}$. In this case, define $h(G)=h(N)+h(G/N)$. It was shown in  \cite{MR1143191} using transfinite induction that $h(G)$ is well defined and satisfies:
\begin{enumerate}
\item for $H\leq G$, we have $h(H)\leq h(G)$;
\item for $N \triangleleft G$, we have $h(G)=h(N)+h(G/N)$;
\item for $G$, we have $h(G) = \sup \lbrace h(F) \mid F\leq G$ finitely generated $\rbrace$.
\end{enumerate}

We can now extend the asymptotic dimension bound from solvable groups to elementary amenable groups. 

\begin{thm}\label{thm:elem-induction}
Let $G$ be an elementary amenable group. Then $\asdim(G)\leq h(G)$.
\end{thm}
\begin{proof}
This follows from Corollary \ref{cor:ext} and induction on the Hirsch length, but we include the details because they illustrate the scheme of the proof that we use later in the context of box spaces.

Fix a proper, right invariant metric on $G$. An immediate consequence of \cite[Theorem 3.2]{DS2006} is that the asymptotic dimension of a virtually abelian group $A$ (i.e.~ a group in $\mathfrak{X}_{1}$) is equal to its Hirsch length. From the definition of Hirsch length for $\mathbf{Elem}$ we have $h(G) = h(N)+h(G/N)$ for any $N\triangleleft G$ and $G \in \mathbf{Elem}$, whereas we have $\asdim G \leq \asdim N + \asdim G/N$. The result now follows using (transfinite) induction based on which $\mathfrak{X}_{\alpha}$ the group $G$ belongs.
\end{proof}

We would also like to remark that a different approach to this result goes through a result of Hillman and Linnell \cite{MR1143191}, where they prove:

\begin{thm}\label{thm:Hillman-Linnell}
Let $G$ be a elementary amenable group of finite Hirsch length and let $\Lambda(G)$ be the maximal normal locally finite subgroup of $G$. Then $G/\Lambda(G)$ is virtually solvable with Hirsch length equal to the Hirsch length of $G$.\qed
\end{thm}

This result illustrates that the finite Hirsch length elementary amenable case is not far from the virtually solvable case.

\section{Box spaces}\label{sec:preliminaries-coarse-geometry}

In this section we make precise the notion of the box space of a (residually finite) group. The most concise definition uses the language of entourage sets for coarse structures introduced by Roe \cite[Section 2.1]{Roe2003}.

\begin{defn}\label{defn:box-space}
 Let $G$ be a countable discrete group and $\sigma = (G_{\alpha})_{\alpha}$ a family\footnote{Here we allow the same subgroup appear multiple times. Thus $\sigma$ may be considered as a \emph{multiset}.} of finite-index subgroups of $G$. Then 
 \begin{enumerate}
  \item The \emph{box space} of $G$ with respect to $\sigma$, denoted by $\square_\sigma G$, is the coarse space whose underlying set is $\bigsqcup_{G_{\alpha} \in \sigma} G / G_{\alpha} $ and whose coarse structure is the minimal connected $G$-invariant coarse structure containing the $G$-diagonals $\Delta_{g}:= \lbrace (x,gx) \mid x \in \bigsqcup_{G_{\alpha} \in \sigma} G / G_{\alpha} \rbrace$ for all $g \in G$. 
  \item The \emph{total box space} of $G$ with respect to $\sigma$, denoted by $\widetilde{\square}_\sigma G$, is the coarse space whose underlying set is $\bigsqcup_{G_{\alpha} \in \sigma} G / G_{\alpha} $ and whose coarse structure is $G$-invariant coarse structure generated by the $G$-diagonals $\Delta_{g}$ for all $g \in G$ (in particular, as a coarse space it is the total space of the family of all finite Schreier quotients of $G$ following the notation and ideas of \cite{EG-permanence}).
 \end{enumerate}
\end{defn}

\begin{rmk}\label{rmk:metrising-box-space}
 It is often easier to work with metric spaces than abstract coarse spaces. Luckily, box spaces have good metrisability.
 \begin{enumerate}
  \item $\widetilde{\square}_\sigma G$ can always be metrised. In fact, if we fix a proper right-invariant metric on $G$, then $\widetilde{\square}_\sigma G$ can be metrised by the metric that is the quotient metric within each finite quotient $G/G_{\alpha}$, and $\infty$ between different finite quotients.
  \item On the other hand, $\square_\sigma G$ can be metrised if and only if $\sigma$ is countable (c.f.~ \cite[Theorem 2.55]{Roe2003}). When $\sigma$ is finite, $\square_\sigma G$ may be realized as a bounded metric space. When $\sigma$ is countably infinite, we may construct the metric as follows: we first index $\sigma$ by the natural numbers, i.e.~ $\sigma = (G_n)_{n \in \N}$, and then take an increasing sequence $\lambda = (\lambda_n)_{n \in \N}$ of positive numbers approaching $\infty$. As before, we also fix a proper right-invariant metric $\dist_G$ on $G$, which determines quotient metrics $\dist_{G/G_n}$ on $G/G_n$. Now for any $x, x' \in \square_\sigma G$, supposing $x \in G/G_n$ and $x' \in G/G_{n'}$ with $n \leq n'$, we define the metric  $\dist_{\sigma, \lambda, \dist_G}$ on $\square_\sigma G$ by
  \[
   \dist_{\sigma, \lambda, \dist_G} (x,x') := \begin{cases}
                   \dist_{G/G_n} (x,x') & \mathrm{if}\ n=n' \; ;\\
                   \sum_{k=n+1}^{n'} \lambda_k & \mathrm{if}\ n<n' \; .
                  \end{cases}
  \]
 Naturally there is a lot of flexibility in how we metrise $\square_\sigma G$, but the crucial property to achieve is that $d(G/G_{m},G/G_{n})$ tends to infinity as $m+n$ does. This behaviour is a consequence of the coarse connectedness condition.
 \end{enumerate}
\end{rmk}

Sometimes it is more convenient to talk about families of coarse metric spaces. For this purpose, we define:

\begin{defn}
 Let $G$ be a countable discrete group and $\sigma = (G_{\alpha})_{\alpha}$ a family of finite-index subgroups of $G$. Fix a proper right-invariant metric $\dist_G$ on $G$, and induce, for each $G_{\alpha} \in \sigma$, the quotient metric $\dist_{G/G_{\alpha}}$. Then the \emph{box family} of $G$ with respect to $\sigma$ and $\dist_G$, denoted by $\widehat{\square}_{\sigma, \dist_G} G$, is the family of metric spaces $\lbrace ( G/G_{\alpha}, d_{G/G_{\alpha}} ) \rbrace_{G_{\alpha} \in \sigma}$.
\end{defn}

For readers familiar with the abstract notion of \emph{coarse families} (c.f. \cite[Section 4]{EG-permanence}), we remark that the box family may also be defined in a similar way as $\widetilde{\square}_\sigma G$ in Definition \ref{defn:box-space}. This way we deduce that as a coarse family, the box family of $G$ with regard to $\sigma$ does not depend on the choice of $\dist_G$. From now on, unless we want to highlight the role of $\dist_G$, we are going to drop it and write simply $\widehat{\square}_\sigma G$ for the box family.

\subsection{Remarks on uniform coarse properties}
We have now introduced three closely related objects attached to a group $G$ together with a family $\sigma$ of subgroups. It is of no surprise that they share a great deal of coarse information. In particular, it is important to know that they have the same asymptotic dimension. To see this, we recall a definition of Guentner \cite[Section 4]{EG-permanence}.

\begin{defn}
 Let $\mathcal{X}:=\lbrace ( X_{\alpha}, d_{\alpha} ) \rbrace_{\alpha}$ be a family of metric spaces. The \emph{total space} $T(\mathcal{X})$ of $\mathcal{X}$ is defined to be the coarse space with underlying set $\sqcup_{\alpha} X_{\alpha}$ and equipped with the disjoint union metric coarse structure, which is generated by the entourage sets $\bigsqcup_{\alpha} \Delta_{R,\alpha} \subset T(\mathcal{X}) \times T(\mathcal{X})$ for all $R > 0$, where $\Delta_{R,\alpha} := \lbrace (x,x') \in X_\alpha \times X_\alpha \ | \ d_\alpha (x, x') \leq R \rbrace $ for each index $\alpha$.
\end{defn}

This choice of coarse structure is exactly the same as metrising each component using the metric analogous to what occurs in Remark \ref{rmk:metrising-box-space}(1). Hence the total space of the box family $\widehat{\square}_\sigma G$ is precisely the extended box space $\widetilde{\square}_{\sigma} G$. The main point of introducing the total space in this generality concerns uniform coarse properties and applications, in this uniform setting, of permanence results in the literature.  

\begin{prop}\label{prop:asdim-coarse-union}
 Let $G$ be a discrete group and $\sigma$ a family of finite index subgroups of $G$. Then 
 \[
  \asdim(\widetilde{\square}_{\sigma}G) = \asdim({\square}_{\sigma}G) = \asdim(\widehat{\square}_{\sigma}G) \; .
 \] 
\end{prop}

\begin{proof}
 The inequalities $\asdim(\widetilde{\square}_{\sigma}G) \leq \asdim(\widehat{\square}_{\sigma}G)$ and $\asdim({\square}_{\sigma}G) \leq \asdim(\widehat{\square}_{\sigma}G)$ follows from the subspace permanence of asymptotic dimension (\cite[Theorem 6.2]{EG-permanence}), while the reverse inequalities follows from the union permanence of asymptotic dimension (\cite[Theorem 6.3]{EG-permanence}).
\end{proof}

The above proposition may also be proved directly from the definition of asymptotic dimension. 

\begin{rmk}\label{rmk:box-space-multiset}
 The asymptotic dimension of $\widehat{\square}_{\sigma}G$ does not change when some of its components are repeated (and thus also when some are coalesced). We observe that the same is true for $\asdim(\widetilde{\square}_{\sigma}G)$ and $\asdim(\square_{\sigma}G)$ by Proposition \ref{prop:asdim-coarse-union}.
\end{rmk}

Although when the group $G$ is not finitely generated (e.g.~ $\bigoplus_{\infty}\Z$), the family $\sigma$ may contain uncountably many different finite-index subgroups, the next proposition shows that it suffices to consider countable subfamilies.

\begin{prop}\label{prop:box-space-countable}
 Let $G$ be a discrete group and $\sigma$ a family of finite index subgroups of $G$. Then
 \[
  \asdim({\square}_{\sigma}G) = \sup_{\mathrm{countable}\ \tau \subset \sigma} \asdim({\square}_{\tau}G) \; .
 \]
 Moreover, if $\sigma$ is nested, then the above supremum may be taken over all nested countable subfamilies $\tau \subset \sigma$. 
\end{prop}

\begin{proof}
 This is a direct consequence of Lemma \ref{lem:asdim-countable} and Proposition \ref{prop:asdim-coarse-union}. The nested case follows from the fact that any countable subset of a nested set is contained in a countable nested subset. 
\end{proof}

The definition of asymptotic dimension given in Definition \ref{defn:asymptotic-dimension} deals with families of metric spaces by asking for the conditions to hold \textit{uniformly} throughout the family. The primary application of this we have seen so far is in Theorems \ref{prop:cf} and \ref{prop:cf2}, where a family of metric spaces played the role of the coarse ``kernel" of a Lipschitz (or large scale uniform) map. However, we can also apply this to the situation where we have a large scale uniform map of \textit{families}. Given a map of families (with the same indexing set)
\begin{equation*}
\lbrace f_{\alpha} : X_{\alpha} \rightarrow Y_{\alpha}\rbrace_{\alpha}
\end{equation*}
we observe that the family of coarse fibres associated to the induced map between total spaces $f= \bigsqcup_{\alpha} f_{\alpha} :T(\mathcal{X}) \rightarrow T(\mathcal{Y})$ is the union of the families of the coarse fibres associated to the individual $f_{\alpha}$'s.

In this situation, using the remarks prior to Corollary 4.12 from \cite{Brodskiy01062008}, there is a uniformised version of Theorem \ref{prop:cf2}:

\begin{thm}\label{prop:cf3}
Let $\lbrace f_{\alpha} : X_{\alpha} \rightarrow Y_{\alpha}\rbrace_{\alpha}$ be a uniformly large scale uniform map of metric families from $\mathcal{X}$ to $\mathcal{Y}$. Suppose that, for every $R>0$, the family of coarse fibres $\bigsqcup_{\alpha} \lbrace f_{\alpha}^{-1}(B_{R}(y)) \rbrace_{y \in Y}$ has asymptotic dimension $\leq n$ uniformly. Then:
\begin{equation*}
\asdim (T(X)) \leq \asdim (T(Y)) + n. 
\end{equation*}\qed
\end{thm}
 
\subsection{Separating families}

We are particularly interested in box spaces that remember enough information about the original group. This would not be possible if the group does not have enough finite-index subgroups to work with. 
 
\begin{defn}\label{defn:separating}
 A family $\sigma$ of finite-index subgroups of $G$ is called \emph{separating} if for any non-empty finite subset $F \subset G \setminus \{1\}$, there is $G_{\alpha} \in \sigma$ such that $g \not\in G_{\alpha}$ for any $g \in F$.

 A group $G$ is called \emph{residually finite} if it has a separating family of finite-index subgroups. Equivalently, the family of \emph{all} finite-index subgroups of $G$ is separating. 
\end{defn}

The box space of a group $G$ with respect to a separating family $\sigma$ may remember a lot of information about $G$, as in the least, one can easily see that the family of quotient maps $\{G \to G/G_{\alpha} \}_{G_{\alpha} \in \sigma}$ has the following joint injectivity property: for any finite subset $F \subset G$, there is $G_{\alpha} \in \sigma$ such that the quotient map $G \to G/G_{\alpha}$ is injective on $F$. However, it may still not behave well enough in terms of the \emph{coarse geometric} information of $G$. For this we need a slightly stronger condition.

\begin{defn}\label{defn:semi-conjugacy-separating}
 A family $\sigma$ of finite-index subgroups of $G$ is called \emph{semi-conjugacy-separating} if for any non-empty finite subset $F \subset G$, there is $G_{\alpha} \in \sigma$ such that $G_{\alpha} \cap g^G G_{\alpha} = \varnothing$ for every $g \in F \setminus \{1\}$, where $g^G$ is the conjugacy class of $g$.
\end{defn}
 
\begin{rmk}
 We remark that when $\sigma$ is nested, both separating family (Definition \ref{defn:separating}) and semi-conjugacy-separating family (Definition \ref{defn:semi-conjugacy-separating}) can be defined by considering only the singleton finite sets $F=\lbrace g \rbrace$. In particular, the situation where $\sigma$ is a chain is discussed in \cite{SWZ2014}. 
\end{rmk}

\begin{lem}\label{lem:semi-conj-separating}
 Let $G$ be a group and $\sigma$ a collection of finite-index subgroups of $G$. Then the following are equivalent:
 \begin{enumerate}
  \item $\sigma$ is semi-conjugacy-separating.
  \item \label{lem:semi-conj-separating:trivial-intersection}For any finite subset $F \subset G$, there is $G_{\alpha} \in \sigma$ such that  $\bigcup_{k \in G} k G_{\alpha} k^{-1} \cap F \subset \{1\}$.
  \item \label{lem:semi-conj-separating:unbounded-injective-radii} For any finite subset $F \subset G$, there is $G_{\alpha} \in \sigma$ such that the quotient map $G \to G / G_{\alpha}$ is injective on $F \cdot k$ for any $k \in G$.
 \end{enumerate}
\end{lem}

\begin{proof}
 (1) $\Leftrightarrow$ (2): (2) can be rewritten as: for any finite subset $F \subset G$, there is $G_{\alpha} \in \sigma$ such that for any $k \in G$ and $g \in F \setminus \{1\}$, $g \not \in k G_{\alpha} k^{-1}$. But the formula at the end is equivalent to $k^{-1} g k \not \in G_{\alpha}$, which is in turn equivalent to $k^{-1} g k G_{\alpha} \cap G_{\alpha} = \varnothing$. Thus the entire statement is equivalent to (1).
 
 (2) $\Leftrightarrow$ (3): (3) can be rewritten as: for any finite subset $F \subset G$, there is $G_{\alpha} \in \sigma$ such that for any $g_1, g_2 \in F$ and $k \in G$, if $g_1 ^{-1} g_2 \not = 1$, then $g_1 ^{-1} g_2 k \not \in kG_{\alpha} $. But the last formula is equivalent to $g_1 ^{-1} g_2 \not \in kG_{\alpha} k ^{-1} $. Thus the statement becomes that for any finite subset $F \subset G$, there is $G_{\alpha} \in \sigma$ such that for any $g \in F^{-1}F \setminus \{1\}$ and $k \in G$, we have $g \not \in kG_{\alpha} k ^{-1}$, which is then equivalent to (2).  
\end{proof}

Once we fix a proper right-invariant metric on $G$, the last condition in the above lemma says that the family of maps $\{ G \to G / G_{\alpha} \}_{G_{\alpha} \in \sigma}$ achieves arbitrarily large injective radii \textemdash~ we simply set $F = B_R(1)$ for an arbitrarily large positive number $R$. This reveals the importance of this condition regarding the coarse geometric information of the box space and the group.

On the other hand, condition (2) directly implies the following. 

\begin{cor}
 Every separating family of finite-index normal subgroups is semi-conjugacy-separating. The same is true for any family $\sigma$ of finite-index subgroups satisfying that there exists a separating family $\tau$ of finite-index normal subgroups such that for any $G_\beta \in \tau$ there is $G_{\alpha} \in \sigma$ such that $G_{\alpha} < G_\beta$. \qed
\end{cor}

\begin{cor}
 A group is residually finite if and only if it has a semi-conjugacy-separating family of finite-index subgroups.
\end{cor}
\begin{proof}
 The ``if'' direction is trivial. The ``only if'' direction holds because every residually finite group has a separating family of finite-index normal subgroups.
\end{proof}

For us, the condition of large injective radii is crucial, as it allows us to perform lifts of uniformly bounded covers. Let us first state a simple lemma (c.f.~ \cite{SWZ2014}) that we are going to use in the next proposition.

\begin{lem}\label{lem:injectivity-radius-isometry-radius}
 Let $(X, \dist_X)$ be an $H$-space, where $H$ acts properly by isometries and $Y$ be the quotient of $X$ by the $H$ action. Let $\phi: X \to Y$ be the quotient map and $\dist_Y$ is the quotient metric of $\dist_X$ under $\phi$. Let $x \in X$ and $R> 0$ be such that $\phi$ is injective when restricted to $B_{3R}(x)$. Then $\phi$ is an isometry when restricted to $B_{R}(x)$.\qed
\end{lem}

We need this lemma to prove:
 
\begin{prop}\label{prop:compare-asdim-box-space-with-group}
 Let $G$ be a residually finite group. Consider the action of $G$ on itself by right multiplication and fix a proper right-invariant metric. Let $\sigma$ be a semi-conjugacy-separating family of finite-index subgroups of $G$. Then $\asdim(G) \leq \asdim(\square_{\sigma} G)$.
\end{prop}

\begin{proof}
 It suffices to show that if $\asdim(\widetilde{\square}_{\sigma} G) \leq n$ for some $n \in \Z^{\geq 0}$, then $\asdim(G) \leq n$. Let us fix a proper right-invariant metric $d_G$ on $G$, and induce a total-box-space metric $d_{\widetilde{\square}_{\sigma} G}$ on $\widetilde{\square}_{\sigma} G$ that restricts to the corresponding quotient metric of $d_G$ on each finite quotient (c.f. Remark \ref{rmk:metrising-box-space}). Given any $R > 0$, by the definition of the asymptotic dimension (\ref{defn:asymptotic-dimension}), there exists $S \geq R$ and an $S$-bounded cover $\mathcal{U}$ of $\widetilde{\square}_{\sigma} G$ with multiplicity no more than $n+1$ and Lebesgue number at least $R$. 
 
 Since we assumed $\sigma$ to be semi-conjugacy-separating, by condition (3) of \ref{lem:semi-conj-separating}, there exists $G_{\alpha} \in \sigma$, such that the quotient map $\pi_{G_{\alpha}}: G \to G/ G_{\alpha}$ is injective on $B_{3S}(1) \cdot g = B_{3S}(g)$ for any $g \in G$. By Lemma \ref{lem:injectivity-radius-isometry-radius}, this implies that $\pi_{G_{\alpha}}$ is isometric on $B_{S}(g)$ for any $g \in G$. Let $\mathcal{U}_{\alpha} := \{U \cap \: (G/ G_{\alpha}) \ | \ U \in \mathcal{U}\}$ be the restriction of $\mathcal{U}$ to $G/ G_{\alpha}$. 
 
 We are going to ``lift'' $\mathcal{U}_{\alpha}$ to an $G_{\alpha}$-invariant cover $\mathcal{V}$ of $G$. For each nonempty $U \in \mathcal{U}_{\alpha}$, since $\diam(U) \leq S$, it is contained in $B_S(x_U)$ for some (in fact, any) $x_U \in U$. Let $g_U \in G$ be such that $x_U = \pi_{G_{\alpha}} (g_U)$. Then as $d_{G/G_{\alpha}} := d_{\widetilde{\square}_{\sigma} G} |_{G/G_{\alpha}}$ is the quotient map of $d_G$, we have 
 \[
  \pi_{G_{\alpha}}^{-1}(B_S(x_U)) = P(\pi_{G_{\alpha}}^{-1}(x_U);S) = P(g_U G_{\alpha};S) = \bigcup_{h \in G_{\alpha}} B_{S}(g_U h)   \; .                                                                                                                                                                                                                                                                                                                      
 \]
 Since our choice of $G_{\alpha}$ ensures that for any different $h, h' \in G_{\alpha}$, $d(g_U h, g_U h') > 3S$, it follows that the right-hand side of the previous long equation is a disjoint union. Hence if we put $\widetilde{U} := \pi_{G_{\alpha}}^{-1} (U) \cap B_{S}(g_U)$, then the pre-image of $U$ under $\pi_{G_{\alpha}}$ decomposes as $\bigsqcup_{h \in G_{\alpha}} \widetilde{U} \cdot h$.
 
 We claim that $\mathcal{V} := \{ \widetilde{U} \cdot h \ | \ U \in \mathcal{U}_{\alpha}, \ h \in G_{\alpha} \}$ is an $S$-bounded cover of $G$ with multiplicity no more than $n+1$ and Lebesgue number at least $R$. Indeed, for any $U \in \mathcal{U}_{\alpha}$ and $h \in G_{\alpha}$, since $\widetilde{U} \cdot h \subset B_{S}( g_U h)$, on the latter of which $\pi_{G_{\alpha}}$ is isometric, it follows that $\diam(\widetilde{U} \cdot h) = \diam (U) \leq S$. To bound the multiplicity, we pick arbitrarily $(n+2)$ mutually different members $\widetilde{U}_0h_0, \widetilde{U}_1h_1, \ldots, \widetilde{U}_{n+1}h_{n+1}$ of $\mathcal{V}$. There are two cases:
 \begin{enumerate}
  \item $U_0, \ldots, U_{n+1}$ are mutually different: then since the multiplicity of $\mathcal{U}$ is at most $n+1$,
  \[
   \pi_{G_{\alpha}} \left( \bigcap_{j=0}^{n+1} \widetilde{U}_j h_j \right) = \bigcap_{j=0}^{n+1} {U}_j = \varnothing
  \]
  which implies $\bigcap_{j=0}^{n+1} \widetilde{U}_j h_j = \varnothing$;
  \item $U_i = U_j$ for some different $i, j \in \{0, \ldots, n+1\}$: then $h_i \not= h_j$ and 
  \[
   \bigcap_{j=0}^{n+1} \widetilde{U}_j h_j \subset  \widetilde{U}_i h_i \cap \widetilde{U}_j h_j = \widetilde{U}_i h_i \cap \widetilde{U}_i h_j = \varnothing \; .
  \]
 \end{enumerate}
 Hence we can bound the multiplicity of $\mathcal{V}$ by $n+1$. Lastly, given any non-empty subset $F \subset G$ with $\diam(F) \leq R$, since $\diam(\pi_{G_{\alpha}}(F)) \leq \diam(F) \leq R$, there is $U \in \mathcal{U}_{\alpha}$ such that $\pi_{G_{\alpha}}(F) \subset U$. In particular, $F \subset \bigsqcup_{h \in G_{\alpha}} \widetilde{U} h \subset \bigsqcup_{h \in G_{\alpha}} B_{S}(g_U h)$. Observe that for any different $h, h' \in G_{\alpha}$, since $d_G(h,h') > 3S$, 
 \[
  d_G(\widetilde{U} h, \widetilde{U} h') \geq d_G(B_{S}(g_U h), B_{S}(g_U h')) > 3S - S -S = S \geq R
 \]
 it follows that $F \subset \widetilde{U} h_0$ for some $h_0 \in G_{\alpha}$, which shows that the Lebesgue number of $\mathcal{V}$ is at least $R$.
\end{proof}

\section{Asymptotic dimension for box spaces of elementary amenable groups}
In this section we provide proofs of the results outlined in the introduction, using the ideas discussed in the previous sections. The basic strategy is to prove an extension formula for the asymptotic dimension of box spaces. To do this, we need first a lemma that provides compatibility between the metrics on quotients and subgroups of a group $G$. 

\begin{lem}\label{key-lemma}
 Let $1 \to N \overset{\iota}{\to} G \overset{\pi}{\to} K \to 1$ be a short exact sequence of discrete groups, where $\iota$ is considered to be an inclusion. Let $H$ be a (not necessarily normal) subgroup of $G$ and $p_G : G \to G/H$ and $p_K : K \to K/\pi (H)$ be the quotient maps. Then
 \begin{enumerate}
  \item there is a unique $G$-equivariant map $\rho: G/H \to K/\pi(H)$ such that $\rho \circ p_G = p_K \circ \pi$, where the action of $g\in G$ on $K/\pi(H)$ is given by left multiplication by $\pi(g)$;
  \item for any $k \in K$, the pre-image of $p_{K}(k) \in K/ \pi(H)$ under $\rho$ is equal to $p_G (g N)$, for any $g \in \pi^{-1}(k)$, and this pre-image is invariant under the action of $N \subset G$ on $G/H$ by left multiplication. 
  \item $[G : H] < \infty$ iff $[K, \pi(H)] < \infty$ and $[N, N \cap H] < \infty$. 
 \end{enumerate}
 Furthermore, let $\dist_G$ be a proper right-invariant metric on $G$. We obtain quotient metrics $\dist_{G/H}$ and $\dist_K$ from $\dist_G$, and then $\dist_{K/\pi(H)}$ from $\dist_K$. Then using the notations in Section \ref{subsection:notations}, we have
 \begin{enumerate}\setcounter{enumi}{3}
  \item $\dist_{K/\pi(H)}$ is the quotient metric of $\dist_{G/H}$ under $\rho$; 
  \item for any $R > 0$, any $k \in K$ and any $g \in \pi^{-1} (k)$, the pre-image of $P(p_{K}(k); R; \dist_{K/\pi(H)})$ under $\rho$ is equal to $p_G ( P(N ; R; \dist_G ) g )$, and moreover this set contains an $R$-net $p_G ( N g )$ isometric to $N / (N \cap g H g^{-1})$ with the quotient metric coming from $\dist_N$.
 \end{enumerate}
\end{lem}

\begin{rmk}
We note that the maps defined in the statement fit into the following diagram:
\[
\xymatrix{
1 \ar@{->}[r] & H\cap N \ar@{->}[r]\ar@{->}[d] & H \ar@{->}[r] \ar@{->}[d]& \pi(H) \ar@{->}[r]\ar@{->}[d] & 1\\
1 \ar@{->}[r] & N \ar@{->}^{\iota}[r] \ar@{->}[d]& G \ar@{->}^{\pi}[r] \ar@{->}^{p_{G}}[d]& K \ar@{->}[r]\ar@{->}^{p_{K}}[d] & 1\\
1\ar@{->}[r] & N/H\cap N \ar@{->}[r] & G/H \ar@{-->}[r]^{\rho} & K/\pi(H) \ar@{->}[r] & 1}
\]
where the final line is not necessarily a short exact sequence of groups, but $G/H \rightarrow K/\pi(H)$ is the quotient map of metric spaces induced by $\pi$, and $N/N\cap H$ is a fibre of this map as a subspace of $G/H$.  
\end{rmk}

\begin{proof}[Proof of \ref{key-lemma}]
Let $N$, $G$, $K$ and $H$ be given as described.
\begin{enumerate}
 \item We define $\rho (g \cdot H) = p_K (\pi (g))$. We need to show it is well defined. Let $g_1, g_2 \in G$ be such that $g_1 ^{-1} g_2 \in H$, then $\pi(g_1) ^{-1} \pi(g_2) \in \pi(H)$, which implies that $p_K(\pi(g_1)) = p_K(\pi(g_2))$. This shows $\rho$ is well defined. It follows immediately from the $G$-equivariance of $p_K$ and $\pi$ that $\rho$ is also $G$-equivariant, while uniqueness follows from the surjectivity of $p_G$.
 \item Given $k \in K$ and $g \in \pi^{-1}(k)$, we see that for any $g' \in G$, the condition $\rho ( p_G(g') ) = p_{K}(k)$ is equivalent to $p_K (\pi (g')) = p_K (\pi (g))$, which is then rewritten as $\pi (g ^{-1} g') \in \pi(H)$, i.e.~ $g ^{-1} g' \in N H$. This is in turn the same as $g' \in g N H$, or equivalently, $g' H \subset g N H$, i.e.~ $p_G(g') \in p_G(g N)$. This proves the first claim, whence the second claim also follows, by the normality of $N$.
 \item This is an immediate consequence of the previous two conclusions, because $[G : H] = |G / H|$, $[K : \pi(H)] = | K/\pi(H)|$, and $[N : N \cap H] = |N / (N \cap H) | = |p_G(N)| = | g \cdot p_G(N)| = | p_G (g N) |$ for any $g \in G$, where we used the action $G \curvearrowright G/H$ by left multiplication.
 \item This follows from the identity of quotient maps $\rho \circ p_G = p_K \circ \pi$.
 \item Let $R,k,g$ be as given. For the first claim, we begin by proving the subset inclusion $p_{G}(P(N;R;d_{g})g) \subset \rho^{-1}(P(p_{K}(k);R;d_{K/\pi(H)}))$. To verify this, it is enough, by (1), to check that $p_{K}(\pi(P(N;R;d_{g})g)) \subset P(p_{K}(k);R;d_{K/\pi(H)})$.
 
Let $g' \in P(N;R;d_{g})g$. As $N$ is a normal subgroup in $G$ we obtain the following inequality:
\begin{equation*}
d_{G}(g',gN)= d_{G}(g',Ng) \leq R.
\end{equation*}
As both $d_{K}$ and $d_{K/\pi(H)}$ are quotient metrics, we have the inequality
\begin{equation*}
d_{K/\pi(H)}(p_{K}(\pi(g')),p_{K}(\pi(g)))\leq R.
\end{equation*}
This last inequality, again by (1), translates into:
\begin{equation*}
d_{K/\pi(H)}(\rho(p_{G}(g'), p_{K}(k)) \leq R.
\end{equation*}
Thus $p_{G}(g') \in \rho^{-1}(P(p_{K}(k);R;d_{K/\pi(H)}))$ as desired.
 
To complete the first part of the claim, we now check 
\begin{equation*}
\rho^{-1}(P(p_{K}(k);R;d_{K/\pi(H)})) \subset p_{G}(P(Ng;R;d_{G}).
\end{equation*}
Let $y \in \rho^{-1}(P(p_{K}(k);R;d_{K/\pi(H)}))$, let $g' \in G$ such that $p_{G}(g') = y$ and let $x = \rho(y)$. From the choices above, we have the following inequality:
\begin{equation*}
d_{K/\pi(H)}(x, p_{K}(k)) \leq R.
\end{equation*}
As both $d_{K}$ and $d_{K/\pi(H)}$ are quotient metrics, this is equivalent to the statement that the cosets  $g'NH$ and $gNH$ are at distance at most $R$ in $G$, and thus, by the normality of $N$ and the right invariance of $d_{G}$:
\begin{equation*}
d_{G}(g'H,gN)=d_{G}(g'NH,gNH)\leq R.
\end{equation*}
Hence, there exists a $h\in H$ such that $g^{'}h \in P(Ng;R;d_{G})$. This completes the proof, as $p_{G}(g'h)=p_{G}(g')=y$, and $P(Ng;R;d_{G}) = P(N;R;d_{G})g$ as the metric $d_{G}$ is right invariant.

 For the second claim, we see that since $N g \subset G$ is an $R$-net of $P(N ; R; \dist_G ) g$ and $\dist_{G/H}$ is the quotient metric of $\dist_G$ under $p_G$, it follows that $p_G ( N g )$ is an $R$-net for $p_G ( P(N ; R; \dist_G ) g )$. Furthermore, the right multiplication by $g$ defines an isometric bijection $\beta_g : G \to G$ that carries $N$ to $Ng$ and is equivariant with respect to the action of $G$ by left multiplication. Since $p_G \circ \beta_g$ maps $N$ to $p_G(Ng)$ equivariantly with respect to the left multiplication by $N$, and this action is transitive on the domain, we see that $p_G \circ \beta_g$ induces canonically a bijection between $p_G(Ng)$ and $N/ ((p_G \circ \beta_g)^{-1} (p_G \circ \beta_g)(1) ) = N / (N \cap g H g^{-1})$. Since $\beta_g$ preserves $d_G$, thus $d_{G/H}$ is also the quotient metric of $d_G$ under $p_G \circ \beta_g: G \to G/H$, which implies that $p_G(Ng)$ and $N / (N \cap g H g^{-1})$ are isometric.
\end{enumerate}
Thus we have proved all the claims.
\end{proof}

With this lemma in hand, we are now able to apply the Hurewicz-type theorems to obtain a permanence result about the asymptotic dimension of box spaces.

\begin{prop}\label{prop:technical-back-end}
 Let $1 \to N \to G \overset{\pi}{\to} K \to 1$ be a short exact sequence of discrete groups. Let $\sigma$ be a family of finite-index subgroups of $G$. Define $\pi(\sigma) := \big( \pi(G_{\alpha}) \big)_{\alpha}$ and $\widehat{\sigma} := \big( N \cap g G_{\alpha} g^{-1} \big) _{G_{\alpha} \in \sigma, g \in G}$. Then 
 \[
  \asdim (\square_\sigma G) \leq \asdim (\square_{\widehat{\sigma}} N) + \asdim (\square_{\pi(\sigma)} K) \; . 
 \]
\end{prop}

Notice that we do not need to assume any of the groups above to be residually finite. 

\begin{proof}
 Fix a proper right-invariant metric on $G$. Let $\widehat{\square}_\sigma G$ be the box family of $G$ with respect to $\sigma$. By Proposition \ref{prop:asdim-coarse-union}, $\asdim (\square_\sigma G) = \asdim (\widehat{\square}_\sigma G) $.
 
 By Lemma \ref{key-lemma}(1) and (4), for every $G_{\alpha} \in \sigma$, there is a $1$-Lipschitz map $\rho_{\alpha}: G/G_{\alpha} \to K/\pi(G_{\alpha})$. Together they form a $1$-Lipschitz family of maps $\rho: \widehat{\square}_\sigma G \rightarrow \widehat{\square}_{\pi(\sigma)} K$. Additionally using Lemma \ref{key-lemma}.(5), we have that for any $R>0$, for every $G_{\alpha} \in \sigma$ and $y \in K / \pi(G_{\alpha})$, the coarse fibre $\rho^{-1} \big( B_R(y; \dist_{K/\pi(G_{\alpha})}) \big)$ contains an $R$-net isometric to $N / (N \cap g G_{\alpha} g^{-1})$, where $g \in (p_K \circ \pi)^{-1} (y)$. Thus by uniform coarse equivalence, we see that 
 \begin{align*}
  & \ \asdim \left( \lbrace \rho^{-1} \big( B_R(y; \dist_{K/\pi(G_{\alpha})}) \big) \rbrace_{G_{\alpha} \in \sigma, y \in K / \pi(G_{\alpha})} \right) \\
  = & \ \asdim \left( \lbrace N / (N \cap g G_{\alpha} g^{-1}) \rbrace_{G_{\alpha} \in \sigma, g\in G} \right) \\
  = & \ \asdim \left( \lbrace N/M \rbrace_{M \in \widehat{\sigma}} \right) \\
  = & \ \asdim ( \widehat{\square}_{\widehat{\sigma}} N ) \; .
 \end{align*}
 
 Applying Theorem \ref{prop:cf3}, we obtain the bound: 
 \[
  \asdim (\widehat{\square}_\sigma G) \leq \asdim ( \widehat{\square}_{\widehat{\sigma}} N ) + \asdim (\widehat{\square}_{\pi(\sigma)} K) \: .
 \]
 The desired bound from the statement of the proposition now follows from Proposition \ref{prop:asdim-coarse-union}.
\end{proof}

One can also prove Proposition \ref{prop:technical-back-end} by using Theorem \ref{prop:cf2} (stated for a single metric space) instead of Theorem \ref{prop:cf3} (stated for a family of metric spaces):

\begin{proof}[Alternative proof of Proposition \ref{prop:technical-back-end}]
 By Proposition \ref{prop:box-space-countable}, it suffices to give a proof when $\sigma$ is countable. Since the case when $\sigma$ is finite is trivial, we assume it is infinite, and write $\sigma = \big( G_n \big)_{n \in \N}$ and $\pi(\sigma) = \big( \pi(G_n) \big)_{n \in \N}$. By Remark \ref{rmk:metrising-box-space}(2), after picking an increasing sequence $\lambda = (\lambda_n)_{n \in \N}$ of positive numbers approaching $\infty$ and a proper right-invariant metric $\dist_G$ on $G$, which induces the quotient metric $\dist_K$ on $K$, we may metrise $\square_\sigma G$ and $\square_{\pi(\sigma)} K$ by $\dist_{\sigma, \lambda, \dist_G}$ and $\dist_{\pi(\sigma), \lambda, \dist_K}$, respectively. 
 
 By Lemma \ref{key-lemma}(1) and (4), for every $G_{n} \in \sigma$, there is a $1$-Lipschitz map $\rho_{n}: G/G_{n} \to K/\pi(G_{n})$. The disjoint union of all the $\rho_{n}$ yields $f: {\square}_\sigma G \rightarrow {\square}_{\pi(\sigma)} K$, which is again $1$-Lipschitz because whenever $x$ and $x'$ come from two different finite quotients $G/G_{n}$ and $G/G_{n'}$ (say $n < n'$), we have by construction that
 \[
  \dist_{\pi(\sigma), \lambda, \dist_K} ( f(x), f(x') ) = \dist_{\sigma, \lambda, \dist_G} (x, x') = \sum_{k = n+1}^{n'} \lambda_k \; .
 \]
 Additionally, for any $R>0$ and any $y \in K / \pi(G_{n}) \subset {\square}_{\pi(\sigma)} K$, the coarse fibre $f^{-1} \big( B_R(y; \dist_{\pi(\sigma), \lambda, \dist_K}) \big)$ decomposes into $X_1 := f^{-1} \big( B_R(y; \dist_{\pi(\sigma), \lambda, \dist_K}) \big) \cap \big(G / G_{n}\big)$ and $X_2 := f^{-1} \big( B_R(y; \dist_{\pi(\sigma), \lambda, \dist_K}) \big) \setminus \big(G / G_{n}\big)$. The above equation implies that $X_2 \subset B_R(x; \dist_{\sigma, \lambda, \dist_G})$ for any $x \in X_1$, whence it follows from applying Lemma \ref{key-lemma}.(5) to $X_1$ that the coarse fibre $f^{-1} \big( B_R(y; \dist_{\pi(\sigma), \lambda, \dist_K}) \big)$ contains an $R$-net isometric to $N / (N \cap g G_{n} g^{-1})$, where $g \in (p_K \circ \pi)^{-1} (x)$. Thus arguing as in the first proof but applying Theorem \ref{prop:cf2} instead of Theorem \ref{prop:cf3}, we obtain $\asdim (\square_\sigma G) \leq \asdim (\square_{\widehat{\sigma}} N) + \asdim (\square_{\pi(\sigma)} K)$. 
\end{proof}

Using Proposition \ref{prop:technical-back-end}, we can prove the theorems stated in the introduction.

\setcounter{thmx}{0}

\begin{thmx}
 Let $1 \to N \to G \to K \to 1$ be a short exact sequence of discrete groups. Then 
 \begin{align*}
  \asdim (\square_\mathrm{F} G) \leq & \ \asdim (\square_\mathrm{F} N) + \asdim (\square_\mathrm{F} K) \; ;  \\
  \asdim (\square_\mathrm{f} G) \leq & \ \asdim (\square_\mathrm{f} N) + \asdim (\square_\mathrm{f} K) \; .
 \end{align*}
\end{thmx}

\begin{proof}
 By the definition of box spaces, for any group $H$, $\asdim (\square_\sigma H) \leq \asdim (\square_\mathrm{F} H)$ for any (not necessarily separating) family $\sigma$ of subgroups of $H$. If in addition $\sigma$ includes only normal groups, then $\asdim (\square_\sigma H) \leq \asdim (\square_\mathrm{f} H)$. 
 
 Let $\sigma$ be the family of all finite-index subgroups of $G$ and apply proposition \ref{prop:technical-back-end}. We see that 
 \begin{align*}
  \asdim (\square_\mathrm{F} G) & \leq \asdim (\square_{\pi(\sigma)} N) + \asdim (\square_{\widehat{\sigma}} K) \\
  & \leq \asdim (\square_\mathrm{F} N) + \asdim (\square_\mathrm{F} K) \; .
 \end{align*}
 The case for $\asdim (\square_\mathrm{f} G)$ is similar, except that we let $\sigma$ be the family of all finite-index \emph{normal} subgroups of $G$, and notice that $\pi(\sigma)$ and $\widehat{\sigma}$ both consist of only normal subgroups.
\end{proof}

\begin{thmx}\label{main-thm:bound-proof}
 For any $G \in \mathbf{Elem_{RF}}$ and any semi-conjugacy-separating family $\sigma$ of finite-index subgroups of $G$, we have 
 \[
  \asdim (G) \leq \asdim(\square_\sigma G) \leq \asdim(\square_{f} G) \leq \asdim(\square_{F} G) \leq h(G).
 \]
\end{thmx}

\begin{proof}
 The first bound connecting the asymptotic dimension of $G$ to that of its box space $\square_{\sigma}G$ is precisely the content of Proposition \ref{prop:compare-asdim-box-space-with-group}. 
 
All that remains is to check $\asdim(\square_{F} G) \leq h(G)$. For any $N, G, K \in \mathbf{Elem_{RF}}$ such that $1 \to N \to G \to K \to 1$ is exact, by Theorem \ref{main-thm:inductive}, we have $\asdim (\square_\mathrm{F} G) \leq  \asdim (\square_\mathrm{F} N) + \asdim (\square_\mathrm{F} K)$, while at the same time, we have $h(G) = h(N) + h(K)$. Therefore by transfinite induction, we get that $\asdim (\square_\mathrm{F} G) \leq h(G)$ for any $G \in \mathbf{Elem_{RF}}$.
\end{proof} 

We now specialise to the class $\mathbf{LFin}\rtimes\mathbf{VPoly}$.

\begin{thmx}
 For any residually finite $G \in \mathbf{LFin}\rtimes\mathbf{VPoly}$ and any semi-conjugacy-separating family $\sigma$ of finite-index subgroups of $G$, we have 
 \[
   \asdim(G) = \asdim(\square_\sigma G) =\asdim(\square_f G) =\asdim(\square_F G) = h(G) \; .
 \]
\end{thmx}

\begin{proof}

The upper bound follows from Theorem \ref{main-thm:bound-proof}. We only need to show that $\asdim G \geq h(G)$, but this follows from \cite[Theorem 3.5]{DS2006} and Proposition \ref{prop:lf-ext}.
\end{proof}

It is known by applying Theorem 3.2 of Gruenberg \cite{MR0087652} that for any polycylic group $H$ and abelian finite group $F$, the reduced wreath product $F\wr H$ is residually finite. Also notice that by the aforementioned theorem of Hillman and Linnell (Theorem \ref{thm:Hillman-Linnell}), we have $h(F\wr H) = h(H)$. Therefore applying Theorem \ref{main-thm:calculate} to $G := F\wr H$, we obtain the following corollary.

\begin{cor}\label{cor:wreath-product}
 Let $F$ be a finite abelian group and $H$ a virtually polycyclic group. Then for any semi-conjugacy-separating family $\sigma$ of subgroups of $F \wr H$, we have
 \[
  \asdim(F \wr H) = \asdim \big(\square_\sigma (F \wr H) \big) = h(F \wr H) = h(H) \; .
 \]\qed

\end{cor}

\bibliographystyle{alpha}
\bibliography{asdim-box-space-alpha}
\end{document}